\documentclass{amsart}[12pt]

\usepackage[utf8]{inputenc}

\usepackage{boxedminipage}
\usepackage{amsfonts}
\usepackage{amsmath} 
\usepackage{amssymb}
\usepackage{graphicx}
\usepackage{amsthm}
\usepackage{t1enc}
\usepackage{subfig}
\usepackage{hyperref}
\usepackage{tabularx}

\usepackage[]{todonotes}
\usepackage[]{algorithm2e}

\newtheorem{theorem}{Theorem}[section]

\usepackage[mathscr]{euscript}
\newcommand{\C}{\mathscr{C}}
\newcommand{\K}{\mathscr{K}}
\newcommand{\M}{\mathscr{M}}

\title{Euclidean volumes of hyperbolic knots}

\author{Nikolay Abrosimov} 
\address{Novosibirsk State University, ul. Pirogova 2, 630090 Novosibirsk, Russia}
\address{Sobolev Institute of Mathematics, pr. Akademika Koptyuga 4, 630090 Novosibirsk, Russia}
\email[]{abrosimov@math.nsc.ru}

\author{Alexander Kolpakov}
\address{Institut de Math\'ematiques, Universit\'e de Neuch\^atel, Rue Emile-Argand 11, 2000 Neuch\^atel, Suisse / Switzerland}
\email[]{kolpakov.alexander@gmail.com}

\author{Alexander Mednykh}
\address{Novosibirsk State University, ul. Pirogova 2, 630090 Novosibirsk, Russia}
\address{Sobolev Institute of Mathematics, pr. Akademika Koptyuga 4, 630090 Novosibirsk, Russia}
\email[]{smedn@mail.ru}

\date{}

\begin{document}

\begin{abstract}
The hyperbolic structure on a $3$--dimensional cone--manifold with a knot as singularity can often be deformed into a limiting Euclidean structure. In the present paper we show that the respective normalised Euclidean volume is always an algebraic number. This stands in contrast to hyperbolic volumes whose number--theoretic nature is usually quite complicated. 
\end{abstract}


\maketitle

\section{Introduction}

The $A$--polynomial of a knot was introduced in \cite{CCGLS} and has become a powerful knot invariant. It encodes not only topological but also geometric information about the knot complement, especially in the case of hyperbolic knots. The notion of $A$--polynomial can be generalised to the case of hyperbolic manifolds with a single cusp \cite{C, CCGLS}. 

The other concept that we deal with, namely degeneration and regeneration of hyperbolic cone-manifold structures, has been studied in many works over the years. In this regard, we refer to the results of Boileau, Leeb, and Porti \cite{BLP-geometrisation, BP-app, Porti-1998}. For some other significant contributions to the problem we also refer to \cite{MaMont, Mont, Weiss-local, Weiss-global}.   

Let $\K$ be a knot (in $\mathbb{S}^3$ or any other closed orientable $3$--manifold $\M$), and let $\C_\alpha = \C_\alpha(\M, \K)$ be the corresponding cone manifold with underlying topological space $\M$ and cone angle $\alpha$ along a singular geodesic in $\M$ isotopic to $\K$. 

We shall assume that $\K$ is a hyperbolic knot (i.e. $\C_0 = \C_0(\M, \K)=\M\setminus\K$ has a complete hyperbolic metric of finite volume) and that a hyperbolic structure on the cone--manifold $\C_\alpha$ exists for any $\alpha \in (\alpha_0 - \varepsilon, \alpha_0)$, where $\alpha_0>\varepsilon>0$, and degenerates (up to rescaling) into a Euclidean structure as $\alpha \rightarrow \alpha_0$.  

On the other hand, given a Euclidean cone manifold $\C_{\alpha_0}$, a hyperbolic or spherical structure can often be ``regenerated'': namely, it will be hyperbolic for $\alpha \in (\alpha_0 - \varepsilon, \alpha_0)$ and spherical for $\alpha \in (\alpha_0, \alpha_0 + \varepsilon)$. Such cone--manifolds exist under some weak cohomological assumptions \cite{Porti-1998}. Also, if the cone angle satisfies $\alpha_0 \leq \pi$ then more stronger results can be established \cite{Porti-Weiss}. 

The hyperbolic volume of $\C_\alpha$ is an important quantity: due to the Mostow--Prasad--Kojima rigidity, the volume of $\C_0$ is a topological invariant whenever it admits a complete hyperbolic metric (of finite volume) \cite{Kojima}. There is also a large number of results concerning rigidity of cone--manifolds in the hyperbolic and other geometries \cite{Porti-Weiss}. 

Since $\C_\alpha$ converges in the Gromov--Hausdorff sense (after an appropriate rescaling) to a Euclidean cone--manifold $\C_{\alpha_0}$, one can define the associated \textit{normalised Euclidean volume} as
$$\mathrm{vol}\,\C_{\alpha_0} = \lim_{\alpha \to \alpha_0} \frac{\mathrm{Vol}\, \C_\alpha}{\ell^3_\alpha}, $$
where $\ell_\alpha = \ell_\alpha(\M, \K)$ is the length of the singular geodesic of $\C_\alpha = \C_\alpha(\M, \K)$.
From here on ``$\mathrm{vol}$'' denotes the normalised volume in contrast to ``$\mathrm{Vol}$'' that refers to the standard one. See \cite[Proposition 4.1]{Porti-1998} (cf. also \cite[\S 8.3--8.4]{Porti-1998}) for more details.

The number--theoretic nature of hyperbolic volume is usually highly intricate \cite{Zagier}. Surprisingly, the normalised Euclidean volume is always \textit{an algebraic number} whose minimal polynomial, in many cases, can be efficiently calculated. The proof of this fact, that has been apparently overlooked before, is the main result of the present paper. 

In the proof we use a modified version of the \textit{A--polynomial}, while the standard one was introduced by Cooper, Culler, Gillet, Long, and Shalen in \cite{CCGLS}. Our version contributes the real length of the singular geodesic instead of the complex one, and will be called the \textit{Riley polynomial} throughout the paper. It appears very suitable for computational purposes. We also provide a pseudocode that computes the minimal polynomial of the normalised volume $\mathrm{vol}\,\C_{\alpha_0}$. This code can be used in any computer algebra system capable of computing resultants and factorising multivariable polynomials, such as SageMath \cite{SageMath} or Mathematica \cite{Mathematica}.  

The main result of the present paper is related to the celebrated theorem by Sabitov that settles the Bellows Conjecture by showing that the volume of a Euclidean polyhedron $P$ is a root of an algebraic equation whose coefficients depend on the edge lengths and combinatorial type of~$P$ (\cite{Sabitov}, see also \cite{Sabitov-UMN, JMS} for a detailed exposition). A higher--dimensional version of Sabitov's theorem and other generalisations were obtained by Gaifullin \cite{Gaifullin1, Gaifullin2}.

\medskip
\noindent
\textbf{Acknowledgements.} The authors would like to thank Louis H. Kauffman for motivating discussions. 

\medskip
\noindent
\textbf{Funding.} N.A. and A.M. were supported by Russian Federation Government (grant 075--15--2019--1613). A.K. was supported by Swiss National Science Foundation (project PP00P2--170560).

\section{Euclidean volumes and algebraic numbers}

Below we shall speak about cone--manifolds $\C_\alpha(\M, \K)$ whose singular set is a knot $\K$ and underlying space is a closed orientable $3$--manifold $\M$, which in concrete instances is often the three--sphere $\mathbb{S}^3$. 

\begin{theorem}
Let $\C_\alpha = \C_\alpha(\M, \K)$ be a cone--manifold with underlying $3$--manifold $\M$ and singular set a knot $\K$ with cone angle $\alpha$. Assume that $\C_\alpha$ has a hyperbolic structure for $\alpha \in (\alpha_0 - \varepsilon, \alpha_0)$ that degenerates into a Euclidean structure as $\alpha \to \alpha_0$. Then the normalised Euclidean volume of $\C_{\alpha_0}$ is an algebraic number. Moreover, its minimal polynomial can be computed from the $A$--polynomial of $(\M,\K)$. 
\end{theorem}

\begin{proof}
The limit $\mathrm{vol}\,\C_{\alpha_0} = \lim_{\alpha \to \alpha_0} \frac{\mathrm{Vol}\, \C_\alpha}{\ell^3_\alpha}$ exists by \cite[Corollary C]{Porti-1998}. It remains to show that its value is among the roots of a polynomial with integer coefficients. 

For this purpose, we express $\mathrm{vol}\,\C_{\alpha_0}$ as a root of the associated Riley polynomial of $(\M, \K)$: this is a ``real'' version of the $A$--polynomial, as opposed to the original ``complex'' one. Let $A(M, L)$ be the $A$--polynomial of $(\M, \K)$ as defined in \cite{CCGLS} that corresponds to the $SL_2(\mathbb{C})$ representations of the fundamental group of $\pi_1(\M \setminus \K)$ (cf. also \cite{C}, where the $PSL_2(\mathbb{C})$ version of the $A$--polynomial is defined).

Following \cite{HLM4}, choose the canonical longitude--meridian pair $(l, m)$ in the fundamental group $\pi_1(\M\setminus\K)$ in such a way that $m$ is an oriented boundary of meridian disc of $\K$ and a longitude curve $l$ is null--homologous outside of $\K$. Let $h : \pi_1(\M\setminus\K) \rightarrow PSL_2(\mathbb{C})$ be the holonomy map of $\M\setminus\K$. Then (see \cite{GM}) $h$ admits two liftings to $SL_2(\mathbb{C})$. The image of $l$ in $SL_2(\mathbb{C})$ under any of these liftings is the same since $l$ is nullhomologous outside the singular set. Thus, up to conjugation in $SL_2(\mathbb{C})$,
\begin{align*}  
h(m)&=\pm\begin{bmatrix}
\exp(i\,\alpha/2) & 0   \\
0 & \exp(-i\,\alpha/2) 
\end{bmatrix},&
h(l)&= \begin{bmatrix}
\exp(\gamma/2) & 0 \\
0 & \exp(-\gamma/2) 
\end{bmatrix},
\end{align*}
where $\gamma = \ell +i \varphi$, $\ell$ is the length of $\K$, and $\varphi$, $- 2\pi \leq \varphi < 2\pi$, is the angle of the lifted holonomy of $\K$. For the sake of simplicity, we will refer to $\gamma = \ell +i \varphi$ as the complex length of the singular geodesic $\K$.

An important property of the $A$--polynomial is that the cone angle $\alpha$ and complex length $\gamma$ of $\K$ are related by the equation $A(L,M) = 0$, where $L=\exp(\gamma/2)$ and $M=\exp(i\,\alpha/2)$. See \cite{CCGLS} for more details.

Let $A(\overline{M}, \overline{L}) = \overline{A(M, L)}$ be the complex conjugate of the initial $A$--polynomial whose coefficients are always integers. Consider \begin{equation*}
    \widehat{A}(M, \overline{L}) = M^{\deg_M A(M,L)} A(M^{-1}, \overline{L}).
\end{equation*}
Once $M=\exp(i\,\alpha/2)$, we have $\overline{M}=M^{-1}$ and $A(\overline{M},\overline{L})=A(M^{-1},\overline{L})$. Also, if $L = \exp(\frac{\ell + i\,\varphi}{2})$ then $\overline{L} = \exp(\frac{\ell - i \,\varphi}{2})$ and the quantity $W = L \overline{L} = \exp( \ell )$ is associated with the real length $\ell$ of the knot $\K$. 

We need to obtain the Riley polynomial that relates the variables $M = \exp(i\,\alpha/2)$ and $W = L \overline{L} = \exp(\ell)$. In order to do this, we consider $L$, $\overline{L}$ and $M$ as independent variables.

Let us compute the consecutive resultants $R_1 = Res_{L}(A(M, L), \widehat{A}(M, \overline{L}))$ and $R_2 = Res_{\overline{L}}(R_1, W - L \overline{L})$, see \cite{Khovanskii-Monin} for basic theory of resultants as applied to Laurent polynomials. As a result, we obtain $R(M, W)$ as a factor of $R_2$ that corresponds to the hyperbolic structure on $\C_{0}$. This factor corresponds to the so-called ``excellent component'' of the character variety of $(\M, \K)$. Note that by construction $R(M, W)$ is a two--variable polynomial with integer coefficients.

The key property of the Riley polynomial is the identity $R(M, W)=0$, whenever $M=\exp(i\,\alpha/2)$ and $W = L \overline{L} = \exp(\ell)$.

As $\alpha \to \alpha_0$, we have that $\ell_\alpha \to 0$ \cite[Corollary C]{Porti-1998}. Thus $M_0 = \exp(i \alpha_0/2)$ is among the roots of $R(M_0, \exp(0)) = R(M_0, 1)$. In particular, $M_0$ is an algebraic number. 

Let us recall the Schl\"{a}fli formula \cite{Hodgson, Milnor,  Vinberg-Geometry-II} that in our case takes the following simple form:
\begin{equation}\label{Schlaefli}
    \mathrm{d} \mathrm{Vol}\, \C_\alpha(\M, \K) = - \frac{1}{2} \, \ell_\alpha \mathrm{d}\alpha.
\end{equation}

Let us observe that the expression for $\mathrm{vol}(\M, \K)$ can be rewritten by using the L'H\^{o}pital rule as follows
\begin{multline}\label{e(MK)}
    \mathrm{vol}\,\C_{\alpha_0} = \lim_{\alpha \to \alpha_0} \frac{\mathrm{Vol}\, \C_\alpha}{\ell^3_\alpha} = \lim_{\alpha \to \alpha_0} \frac{(\mathrm{Vol}\, \C_\alpha)'_\alpha}{(\ell^3_\alpha)'_\alpha} = \lim_{\alpha \to \alpha_0} \frac{-\frac{1}{2} \ell_\alpha}{3 \ell^2_\alpha (\ell_\alpha)'_\alpha} = -\frac{1}{3} \lim_{\alpha \to \alpha_0} \frac{1}{(\ell^2_\alpha)'_\alpha},
\end{multline}
where we use the Schl\"{a}fli formula \eqref{Schlaefli} in order to differentiate $\mathrm{Vol}\,\C_{\alpha}$. Here and below, we shall use $f'_x$ as a shortcut for $\frac{\mathrm{d} f}{\mathrm{d} x}$, for any expression $f$ that depends on a variable $x$ explicilty or implicitly. 

Moreover, as $\alpha\rightarrow\alpha_0$ we have that $\ell=\ell_\alpha \to 0$ and $W = \exp(\ell)\to 1$. Thus we can introduce a new variable $X$ and write $W = 1 + X$, where $X \to 0$ as $\alpha\rightarrow\alpha_0$. Then $\ell = \ln(1 + X)$, $\ell^2 = \ln^2(1 + X)=X^2-X^3+O(X^4)$ as $X\rightarrow 0$. Hence $(\ell^2)'_\alpha = (X^2)'_\alpha + O(X^2)$. By using \cite[Corollary C]{Porti-1998}, we can replace $O(X^2)$ with $O(|\alpha - \alpha_0|)$, as $\alpha \to \alpha_0$. The resulting asymptotic expansion $(\ell^2)'_\alpha = (X^2)'_\alpha + O(|\alpha - \alpha_0|)$, as $\alpha \to \alpha_0$, allows us to use only polynomial expressions in the rest of the proof.

By computing the resultant of $R(M, W)$ with $(W - 1 - X)$ in $W$ first, and then computing one more resultant of the obtained expression with $(Y - X^2)$ in $X$, we find the minimal polynomial $P(M, Y)$ for $Y$ over the ring $\mathbb{Q}[M]$. 

Now let us consider $Y$ as an implicit function $Y = Y(M)$ defined by the equation $P(M, Y(M))=0$ together with the condition $Y(M_0) = 0$, for $M_0 = \exp(i \alpha_0/2)$. This allows us to compute the derivative $Y'(M)$ in terms of $P(M, Y)$ and $Y(M)$ itself. Since $P'_M(M, Y(M)) + P'_Y(M, Y(M)) \, Y'(M) = 0$, let us put $Q(M, Y, Z) = P'_M(M, Y) + P'_Y(M, Y) \, Z$, where $Z = Y'(M)$ is a new variable.

By taking the resultant of $P(M, Y)$ and $Q(M, Y, Z)$ in $Y$, we finally obtain the minimal polynomial $S(M, Z)$ for $Z = (X^2)'_\alpha$ over the ring $\mathbb{Q}[M]$, after choosing the appropriate irreducible factor. 

By passing to the $\frac{d}{dM}$ derivative, we can rewrite \eqref{e(MK)} simply as 
\begin{equation}
    \mathrm{vol}\,\C_{\alpha_0} = \frac{2 i}{3 M_0 Z_0},
\end{equation}
where $Z_0 = \lim_{\alpha \to \alpha_0} Y'(M)$, for $M=\exp(i \alpha/2)$, is among the roots of the polynomial $S(M_0, Z)$, for $M_0=\exp(i \alpha_0/2)$. 

As $M_0$ is an algebraic number and $S(M, Z)$ is a polynomial with integer coefficients, we conclude that $Z_0$ is algebraic. Hence, $\mathrm{vol}\,\C_{\alpha_0}$ is also algebraic. 
\end{proof}

\section{Computing the minimal polynomial for normalised volume}\label{section:algorithm}

In this section we provide a pseudocode that computes the minimal polynomial of $\mathrm{vol}(\M, \K)$ starting from the $SL_2(\mathbb{C})$ $A$--polynomial of $(\M, \K)$ as input. This algorithm can be used in any computer algebra system that has enough functionality in commutative algebra, such as SageMath \cite{SageMath} or Mathematica \cite{Mathematica}. 

\medskip

\begin{algorithm}[H]
 \KwData{$A(M, L)$ = the $A$--polynomial of $(\M, \K)$.}
 \KwResult{The minimal polynomial of $\mathrm{vol}\,\C_{\alpha_0}$.}
\end{algorithm}
\newcounter{No}
\begin{list}{\arabic{No}.}{\usecounter{No} \itemindent=0pt \rightmargin=0em}
    \item Let $d = \mathrm{deg}_M A(M, L)$;\; 
    \item Let $\overline{L}$ be a new variable. Let $R_1, R_2, R_3$ be two auxiliary variable;\; 
    \item Let $W$ be a new variable;\;
    \item Let $\widehat{A}(M, L) := M^d \cdot A(M^{-1}, \overline{L})$;\;
    \item Let $R_1$ be the resultant of $A(M, L)$ and $\widehat{A}(M, L)$ in $L$;\; 
    \item Let $R_2$ be the resultant of $R_1$ and $W - L \cdot \overline{L}$ in $\overline{L}$;\; 
    \item Factorise $R_2$ and isolate its irreducible factor $R(M, W)$ that corresponds to the excellent component of the character variety of $(\M, \K)$;\; 
    \item Let $R_1$ be the resultant of $R(M, W)$ and $W - X - 1$ in $W$;\;
    \item Let $R_2$ be the resultant of $R_1$ and $Y - X^2$ in $X$;\;
    \item Factorise $R_2$ and isolate its irreducible factor that corresponds to the minimal polynomial of $Y$ over $\mathbb{Q}(M)$;\;
    \item Set $Y = Y(M)$ to be a function of $M$. Let $Y' = Y'(M)$ be the derivative of $Y(M)$ with respect to $M$;\;
    \item Differentiate $P(M, Y(M))$ with respect to $M$. Store the output as $R_1$;\;
    \item Substitute $Y'(M)$ in $R_1$ by a new variable $Z$. Store the output as $Q(M, Y, Z)$;\;
    \item Let $R_2$ be the resultant of $P(M, Y)$ and $Q(M, Y, Z)$ in $Y$;\;
    \item Factorise $R_2$ and isolate its irreducible factor $S(M, Z)$ that corresponds to the minimal polynomial of $Z$ over $\mathbb{Q}(M)$;\;
    \item Let $V$ be a new variable. Let $I$ be the complex unit and let $R_1$ be $3 \cdot M \cdot Z \cdot V - 2 \cdot I$. Let $R_2$ be $R(M, 1)$;\;
    \item Let $R_3$ be the resultant of $R_1$ and $S(M, Z)$ in $Z$. Let $R_4$ be the resultant of $R_2$ and $R_3$ in $M$;\;
    \item Factorise $R_4$ and isolate its irreducible factor $F(V)$ that corresponds to the minimal polynomial of $V$ over $\mathbb{Q}$;\;
    \item Output $F(V)$.
\end{list}

\medskip

In the above algorithm, Step 7 is more involved since it needs the understanding of complete hyperbolic structure on $(\M, \K)$ in order to determine the excellent component. For the case of knots in $\mathbb{S}^3$ the excellent component is unique up to conjugation and birational isomorphism \cite{Hilden-et-al}. In other cases, it can often be determined from geometric considerations and computations using Snappy \cite{Snappea}. 

Another task that does not seem amenable to a universal approach is determining the numerical value of $\mathrm{vol}\,\C_{\alpha_0}$ in order to isolate the correct irreducible factor in Step 18. The necessary computation can be achieved by using a known volume formula (cf. \cite{Mednykh, Mednykh-Manilla} for a few suitable ones), or by using Snappy \cite{Snappea} with enough numerical precision.

\section{Two--bridge knot case}

Let $\K$ be a hyperbolic two--bridge knot in $\mathbb{S}^3$ and $\C_\alpha= \C_\alpha(\mathbb{S}^3, \K)$. Then according to \cite{Porti-two-bridge} there exists an angle $\alpha_0 \in [2\pi/3, \pi)$ such that $\C_\alpha$ is hyperbolic for all $\alpha \in [0, \alpha_0)$, Euclidean for $\alpha = \alpha_0$, and spherical for all $\alpha \in (\alpha_0, 2\pi-\alpha_0)$. 

Thus we can define the (normalised) Euclidean volume $\mathrm{vol}(\K)=\mathrm{vol}(\mathbb{S}^3, \K)$ of $\K$. Because of Weiss' rigidity theorem \cite{Weiss-local, Weiss-global}, $\mathrm{vol}(\K)$ will be also a topological invariant of $\K$ together with its hyperbolic volume.

\section{A gallery of examples}

Below we explicitly compute several Euclidean volumes: some for hyperbolic $2$--bridge knots (the figure--eight knot $4_1$ and the 3--twist knot $5_2$) in $\mathbb{S}^3$, and some for their ``sister'' manifolds. The complexity of minimal polynomial for $\mathrm{vol}\,\C_{\alpha_0}$ seems to grow fast when the topological complexity of $(\M, \K)$ increases. The Euclidean volumes for two--bridge knots with less than $8$ crossings were computed in \cite{Mednykh}.

Let $f_1(x_1, x_2, \ldots, x_n), \ldots, f_k(x_1, x_2, \ldots, x_n)$ be a set of polynomials over $\mathbb{Z}$. Let $Res_{x_{i_1}, \ldots, x_{i_{k-1}}}(f_1, \ldots, f_k)$ denote the ``multiple resultant'' inductively defined by 
\begin{equation*}
    Res_{x_{i_1}, \ldots, x_{i_{k-1}}}(f_1, \ldots, f_k) := Res_{x_{i_1}, \ldots, x_{i_{k-2}}}(f_1, \ldots, Res_{x_{k-1}}(f_{k-1}, f_k)).
\end{equation*}
We shall use this notation as a shorthand whenever we need to take several resultants in one scoop. Following \cite{Khovanskii-Monin} it can be extended to Laurent polynomials.

\subsection{Knot $4_1$: the figure--eight}

% this is already computed

The $A$--polynomial of $\K = 4_1$ is given in the KnotInfo census \cite{KnotInfo}
\begin{equation*}
    A(M, L) = M^4 + L (-M^8 + M^6 + 2 M^4 + M^2 - 1) + L^2 M^4.
\end{equation*}

The Riley polynomial $R(M,L)$ is a factor of the resultant of the $A$--polynomial, its conjugate $\overline{A}$ and $W = L \overline{L}$ with respect to the variables $L, \overline{L}$: 
\begin{multline*}
    R(M, L) = Res_{L,\overline{L}}(A(L, M), M^8 A(\overline{L}, M^{-1}), W - L\overline{L}) = -M^8 + (1 - 2 M^2 - 3 M^4 
    \\+ 2 M^6 + 6 M^8 + 2 M^{10} - 3 M^{12} - 2 M^{14} + M^{16}) W - M^8 W^2.
\end{multline*}

First we replace the variable $X$ with $W = 1 + X$, and then once again pass to the new variable $Y$ such that $X = Y^2$ by using resultants. Thus, we obtain
\begin{multline*}
    P(M, Y) = 1 - 4 M^2 - 2 M^4 + 16 M^6 + 9 M^8 - 24 M^{10} - 34 M^{12} + 12 M^{14} + 52 M^{16} \\ + 12 M^{18} - 34 M^{20} - 24 M^{22} 
    + 9 M^{24} + 16 M^{26} - 2 M^{28} - 4 M^{30} + M^{32} 
    \\ + (-1 + 4 M^2 + 2 M^4 - 16 M^6 - 11 M^8 + 28 M^{10} + 40 M^{12} - 16 M^{14} - 60 M^{16} - 16 M^{18} 
    \\ + 40 M^{20} + 28 M^{22} - 11 M^{24} - 16 M^{26} + 2 M^{28} + 
    4 M^{30} - M^{32}) Y + M^{16} Y^2.
\end{multline*}

Taking into account the identity $P'_M(M, Y(M)) + P'_Y(M, Y(M)) \, Y'(M) = 0$, we compute the minimal polynomial $Q(M, Y, Z)$ of $Z = Y'(M)$ over $\mathbb{Z}[Y, M]$ which satisfies $Q(M, Y, Y'(M)) = 0$. This computation yields
\begin{multline*}
    Q(M, Y, Z) = 
    -8 M - 8 M^3 + 96 M^5 + 72 M^7 - 240 M^9 - 408 M^{11} + 168 M^{13} 
    \\+ 832 M^{15} + 216 M^{17} - 680 M^{19} - 528 M^{21} + 216 M^{23} 
    + 416 M^{25} - 56 M^{27} - 120 M^{29} 
    \\+ 32 M^{31} + (8 M + 8 M^3 - 96 M^5 - 88 M^7 + 280 M^9 + 480 M^{11} - 
    224 M^{13} - 960 M^{15} 
    \\- 288 M^{17} + 800 M^{19} + 616 M^{21} 
    - 264 M^{23} - 416 M^{25} + 56 M^{27} + 120 M^{29} - 32 M^{31}) Y 
    \\+ 16 M^{15} Y^2 + (-1 + 4 M^2 + 2 M^4 - 16 M^6 - 11 M^8 + 28 M^{10} + 
    40 M^{12} - 16 M^{14} - 60 M^{16} 
    \\- 16 M^{18} + 40 M^{20} 
    + 28 M^{22} - 11 M^{24} - 16 M^{26} + 2 M^{28} + 4 M^{30} - M^{32}) Z + 2 M^{16} Y Z.
\end{multline*}

The resultant $Res_Y(P(M, Y),\, Q(M, Y, Z))$ has a unique irreducible factor $S(M, Z)$ that defines the value of $Z$ for any fixed $M$. We obtain
\begin{multline*}
    S(M, Z) = -256 M^7 + 256 M^9 - 64 M^{11} - 256 M^{13} + 640 M^{15} - 256 M^{17} - 64 M^{19} 
    \\+ 256 M^{21} - 256 M^{23} + (-16 + 56 M^2 + 24 M^4 - 160 M^6 - 88 M^8 + 168 M^{10} + 160 M^{12} 
    \\- 32 M^{14} + 32 M^{18} - 160 M^{20} - 168 M^{22} + 88 M^{24} + 
    160 M^{26} - 24 M^{28} - 56 M^{30} 
    \\+ 16 M^{32}) Z - M^{17} Z^2.
\end{multline*}

By solving $R(M, 1) = 0$ we obtain $M_0 = \frac{1}{2} + i\,\frac{\sqrt{3}}{2}$, while $S(M_0, Z)$ has $Z_0 = 36 + 12 i \sqrt{3}$ among its roots. Then we obtain the known value \cite{Mednykh, MR}
\begin{equation*}
    \mathrm{vol}(4_1) = \frac{2 i}{3 M_0 Z_0} = \frac{\sqrt{3}}{108}.
\end{equation*}

\subsection{The figure--eight sister: manifold $\mathsf{m003}$}

The $A$--polynomial of a hyperbolic knot $\K$ in a manifold $\M$ different from $\mathbb{S}^3$ can also be defined by using the fact that $\M$ is a complete hyperbolic manifold with a single cusp and by considering its peripheral curve system \cite{C}. In this case, the $A$--polynomial is computed by using the Neumann--Zagier gluing equations: given a set of equations one can reduce them to a single polynomial by excluding the so-called ``shape'' or ``cross--ratio'' parameters. For more details see \cite[\S 2 -- \S 3]{C}, and also \cite[\S 6]{C} for some examples.

Let $\M$ be manifold $\mathsf{m003}$ in the Snappea census \cite{Snappea}, that is the sister manifold of the figure eight knot complement. By using the approach outlined above we obtain
\begin{equation*}
     A(M, L) = M^3 + L^2 (1 - M - 2 M^2 - M^3 + M^4) + L^4 M.
\end{equation*}

Let $\C_\alpha = \C_\alpha(\M, \K)$ be the cone--manifold obtained from $\M$ by the generalised Dehn surgery with parameters $(p, q) = (2\pi/\alpha, 0)$ that produces a cone angle of $\alpha$ along a singular curve that is isotopic to a knot $\K$ in $\M$ (see \cite{NZ}).

Let us note that the cone--manifold $\C_{\alpha}$ is a branched cover of the cone--manifold $6^2_2(\pi, \alpha/2)$ where the branching index is $2$ over \textit{both} singular components of the latter. In contrast, the figure eight cone--manifold $4_1(\alpha)$ is obtained as a double cover of $6^2_2(\pi, \alpha)$ branched over \textit{only} the $\pi$--component. 

The cone--manifold $6^2_2(\alpha, \beta)$, with $0 \leq \alpha, \beta \leq 4\pi/3$, is known to have Euclidean structure whenever $\cos(\alpha/2) + \cos(\beta/2) = 1/2$ and hyperbolic structure for $\cos(\alpha/2) + \cos(\beta/2) > 1/2$ \cite[Proposition 1.3.1]{Mednykh-Manilla}. Thus $\C_{4 \pi/3}$ is a Euclidean cone--manifold, while $\C_\alpha$ is hyperbolic for 
$0 \leq \alpha < 4\pi/3$.

By following the same computational procedure as for the figure eight knot, we obtain
\begin{equation*}
    \mathrm{vol}(\mathsf{m003}) = \frac{2\sqrt{3}}{27}.
\end{equation*}

It is worthy mentioning how $\mathsf{m003}$ can be constructed. First of all, there are exactly two complete hyperbolic manifolds arising from the face identifications of two ideal regular tetrahedra: the figure eight knot complement $4_1$ (numbered $\mathsf{m004}$ in the Snappea census \cite{Snappea}) and its sister $\mathsf{m003}$. Another way to see the relationship between $\mathsf{m003}$ and $\mathsf{m004}$ is by considering the smallest branched covers of the orbifold $6^2_2(\pi, 0)$. It has exactly two degree $2$ manifold covers: $\mathsf{m003}$ (doubly branched over both the cusp and orbifold components) and $\mathsf{m004}$ (which branches only over the orbifold component).

\begin{figure}[t]
    \centering
    \includegraphics[scale=0.20]{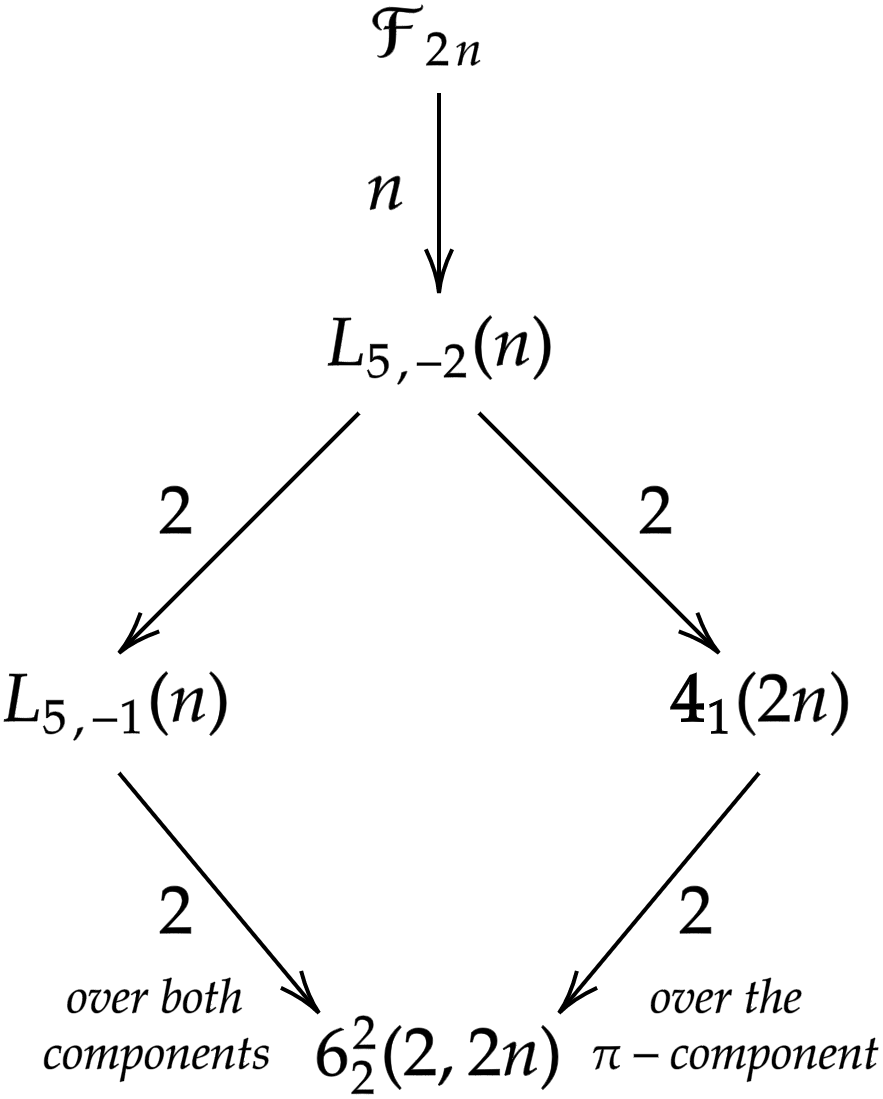}
    \caption{Diagram of orbifold covers for the Fibonacci manifold}
    \label{fig:fibonacci}
\end{figure}

Let us consider the orbifold $L_{p,q}(n)$ obtained by performing $(p, q)$--surgery on one component of the Hopf link, and $(n, 0)$--surgery on the other component. The underlying space of $L_{p,q}(n)$ is the lens space $L_{p,q}$. Let $4_1(n)$ be the figure eight knot orbifold with cone angle $2\pi/n$. Then we obtain the covering diagram in Figure~\ref{fig:fibonacci} (also see \cite{fibonacci2}). 

In the diagram, $\mathcal{F}_{2n}$ denotes the Fibonacci manifold with index $2n$ whose fundamental group has $2n$ generators as defined in \cite{fibonacci} (cf. also \cite{fibonacci3, fibonacci2}). The fundamental groups of such manifolds are the Fibonacci groups defined by Conway \cite{Conway}. These groups have the following presentation: 
\begin{equation*}
    \pi_1(\mathcal{F}_{2n}) \cong \langle x_0, x_2, \ldots, x_{2n-1} \,|\, x_i x_{i+1} = x_{i+2}, \; i\;\mathrm{mod}\;2n \rangle.
\end{equation*}

The limiting case $n \to \infty$ of the covering diagram in Figure~\ref{fig:fibonacci} gives the two manifold covers of $6^2_2(\pi, 0)$ mentioned above: $L_{5,-1}(\infty) = \mathsf{m003}$ and $4_1(\infty) = \mathsf{m004}$.

Despite $\mathsf{m003}$ and $\mathsf{m004}$ having equal hyperbolic volumes, the Euclidean cone--manifold deformations distinguish these sister manifolds. The same holds true for volumes of their hyperbolic Dehn fillings.

\subsection{Knot $5_2$: the triple--twist knot}

The $A$--polynomial of $\K = 5_2$ can be found in the KnotInfo census \cite{KnotInfo}
\begin{multline*}
    A(M, L) = 1 + L (-1 + 2 M^2 + 2 M^4 - M^8 + M^{10}) \\
 + L^2 (M^4 - M^6 + 2 M^{10} + 2 M^{12} - M^{14}) + L^3 M^{14}.
\end{multline*}

By performing the algorithm described in Section~\ref{section:algorithm} we obtain that the normalised Euclidean volume has numerical value
\begin{equation*}
\mathrm{vol}(5_2) \approx 0.009909630999945638.
\end{equation*}
This number is algebraic with minimal polynomial
\begin{equation*}
    785065068490752\, x^8 + 412091172864\, x^6 + 64457856\, x^4 - 864\, x^2 - 1.
\end{equation*}
Earlier, the numerical value of  $\mathrm{vol}(5_2)$ was found in \cite{S}. The above minimal polynomial was firstly computed in \cite{Mednykh}.

\subsection{Sisters of $5_2$: the manifolds $\mathsf{m016}$ and $\mathsf{m017}$}

Similar to the case of figure--eight and its sister manifold, knot $5_2$ and manifold \textsf{m017} from the Snappea census \cite{Snappea} are related to the covering diagram in Figure~\ref{fig:cover-2}.

\begin{figure}[ht]
    \centering
    \includegraphics[scale=0.20]{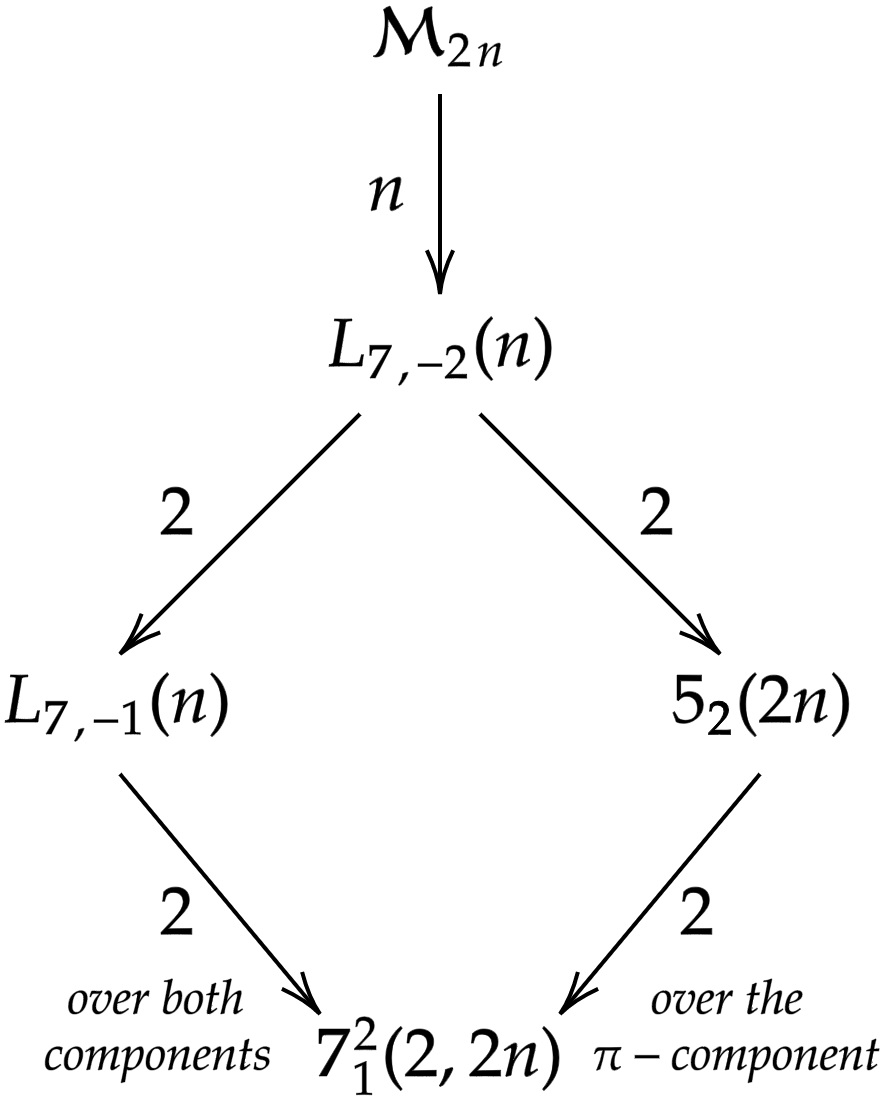}
    \caption{Diagram of orbifold covers for the Bandieri--Kim--Mulazzani manifold}
    \label{fig:cover-2}
\end{figure}

The diagram in Figure~\ref{fig:cover-2} uses the following notation: $5_2(n)$ is the orbifold obtained from by $(n,0)$--surgery on knot $5_2$, $7^2_1(2,n)$ is the orbifold with singular set link $7^2_1$ with orbifold singularities of angles $\pi$ and $2\pi/n$ on the respective components, and $L_{p,q}(n)$, as before, denotes the result of $(p,q)$--surgery on one component of the Hopf link and $(n,0)$--surgery on the other. The manifold $\mathcal{M}_{2n}$ is the $2n$--fold cyclic covering of $5_2(2n)$ described in \cite{BKM-5_2}. The fundamental group of $\mathcal{M}_{2n}$ has the following cyclic presentation with $2n$ generators:
\begin{equation*}
    \pi_1(\mathcal{M}_{2n}) \cong \langle x_0,\ldots x_{2n-1} \,|\, x_{i+1}^{-1}x_{i}x_{i+2}x_{i+1}^{-1}x_{i}x_{i+1}^{-1}x_{i+2}=1, \; i\;\mathrm{mod}\;2n \rangle.
\end{equation*}

The orbifold $7^2_1(\pi, 0)$ has two manifold covers: one is the knot $5_2$ complement, and the other is manifold \textsf{m017}. Thus we obtain sister manifolds again. However, in contrast to the figure--eight knot case, knot $5_2$ has one more sister: manifold \textsf{m016} in the Snappea census \cite{Snappea}. The manifolds \textsf{m016} and \textsf{m017} share the $A$--polynomial, and are distinguished by their Alexander polynomials. However, the geometric relation of \textsf{m017} to its sisters remains unknown to us.

Let $7^2_1(\pi, \alpha)$ be the cone--manifold with singularity $7^2_2$ link and underlying topological space $\mathbb{S}^3$, having cone angles $\pi$ and $\alpha$ along its respective components. Let $5_2(\alpha)$ be the cone--manifold with singularity $5_2$ knot and underlying topological space $\mathbb{S}^3$. Also, let $\C_\alpha$ denote the $(2\pi/\alpha, 0)$--surgery on $\mathsf{m017}$ manifold resulting in a cone--manifold with underlying space $\mathsf{m017}$ and cone angle $\alpha$.

Then $5_2(\alpha)$ is a degree $2$ cover of $7^2_1(\pi, \alpha)$ branched along its $\pi$--component, while $\C_\alpha$ doubly covers $7^2_1(\pi, \alpha/2)$. The latter covering is branched along both singular components.

The cone--manifold $7^2_1(\pi, \alpha_0)$ has Euclidean structure with cone angle 
\begin{equation*}
    \alpha_0 \approx 2.4071698135544546,
\end{equation*}
such that $M_0 = \exp(i \alpha_0/2)$ has the following minimal polynomial:
\begin{equation*}
    1 - 2 M_0 - M^2_0 + 8 M^3_0 - 11 M^4_0 + 8 M^5_0 - M^6_0 - 2 M^7_0 + M^8_0.
\end{equation*}

Then $5_2(\alpha_0)$ is a Euclidean cone--manifold with cone angle $\alpha_0$ along the knot $5_2$ in $\mathbb{S}^3$, and $\C_{2 \alpha_0}$ is a Euclidean cone--manifold with cone angle $2\alpha_0$ along a singular geodesic in $\mathsf{m017}$. The Euclidean cone--manifold structure on $5_2(\alpha_0)$ was previously studied in \cite{S} by the method developed in \cite{MR-Bielefeld}.

The $A$--polynomial of $\mathsf{m017}$ is computed by using the approach of \cite{C}:
\begin{multline*}
     A(M, L) = - M^5 + L^2 (M - M^2 + 2 M^4 + 2 M^5 - M^6) \\
     + L^4 (1 - 2 M - 2 M^2 + M^4 - M^5) + L^6 M.
\end{multline*}

The normalised Euclidean volume obtained from it equals
\begin{equation*}
    \mathrm{vol}(\mathsf{m017}) \approx 0.0792770479995651
\end{equation*}
and has minimal polynomial
\begin{equation*}
    191666276487\, x^8 + 6438924576\, x^6 + 64457856\, x^4 - 55296\, x^2 - 4096. 
\end{equation*}

Since $\mathsf{m016}$ and $\mathsf{m017}$ share the $A$--polynomial, it is natural to expect that $\mathsf{m016}$ admits a Euclidean structure and has the same normalised Euclidean volume as $\mathsf{m017}$.

\end{document}